\documentclass[12pt]{amsart}
\usepackage{amsthm} 

\usepackage[latin1]{inputenc}
\usepackage{subfigure}
\usepackage{color}
\usepackage{amsmath}
\usepackage{amsthm}
\usepackage{amstext}
\usepackage{amssymb}
\usepackage{amsfonts}
\usepackage{graphicx}
\usepackage{young}
\usepackage{multicol}
\usepackage{mathrsfs}
\usepackage{stmaryrd}
\usepackage{bbm}
\usepackage[all]{xy}
\usepackage{mathabx}
\usepackage{tikz}
\usepackage{mathtools}
\usepackage{tabularx}
\usepackage{array}
\usepackage{commath}

\usepackage{caption}
\usepackage[normalem]{ulem}

\theoremstyle{plain}
\theoremstyle{definition}
\newtheorem{theorem}{Theorem}[section]

\newtheorem{lemma}[theorem]{Lemma}
\newtheorem{definition}[theorem]{Definition}

\newtheorem{example}[theorem]{Example}

\DeclareMathAlphabet{\mathpzc}{OT1}{pzc}{m}{it}
\newcommand{\Ext}{\operatorname{Ext}}

\newcommand{\grlex}{{\textit grevlex}}
\usepackage{amssymb,amsmath,tabularx,graphicx}
\usepackage{tikz}
\usetikzlibrary[calc,intersections,through,backgrounds,arrows,decorations.pathmorphing]

\DeclareMathOperator*{\SL}{SL}

\begin{document}

\title{Leading Terms of $\SL_3$ Web Invariants}

\author{V\'eronique Bazier-Matte}
\address{Laboratoire de Combinatoire et d'Informatique Math\'ematique,
Universit\'e du Qu\'ebec \`a Montr\'eal}
\email{bazier-matte.veronique@courrier.uqam.ca, douville.guillaume@courrier.uqam.ca,  alexander.garver@gmail.com, patriasr@lacim.ca, hugh.ross.thomas@gmail.com, emineyyildirim@gmail.com}
\author{Guillaume Douville}
\author{Alexander Garver}
\author{Rebecca Patrias}
\author{Hugh Thomas}
\author{Emine Y{\i}ld{\i}r{\i}m}

\begin{abstract}We use Khovanov and Kuperberg's web growth rules to identify the minimal term in the invariant associated to an $\text{SL}_3$ web diagram, with respect to a particular term order.
\end{abstract}

\maketitle

\section{Introduction}
Let $V$ be a three-dimensional vector space over $\mathbb C$, and let $V^*$ be its dual space.  Write $(V^*)^a \times V^b$ for the product of $a$ copies of $V^*$ and $b$ copies of $V$.  The coordinate ring of $(V^*)^a \times V^b$ is a polynomial ring over $\mathbb C$ in $3(a+b)$ variables, on which $\SL(V)$ acts naturally.  We will be interested in the subring of polynomials that are invariant under this action.

Kuperberg \cite{kuperberg1996spiders} describes a certain basis for this ring --- in fact, more than one basis.  Place $a$ white dots and $b$ black dots around a circle in some order.  Each white dot is associated with a 
particular copy of $V^*$ and each black dot is associated with a particular copy of $V$.  
The invariants are spanned by certain ``tensor diagrams", which are bipartite graphs whose bipartition extends to the fixed coloring on the boundary, such that every internal vertex has degree 3.  The tensor diagrams are not linearly independent, but Kuperberg identifies a spanning subset of ``web diagrams" that are linearly independent.  

Interest in this basis was recently reawakened thanks to the work of Fomin and Pylyavskyy \cite{fomin2016tensor}.  They constructed a cluster algebra structure on the ring of invariants that interacts well with the web basis
(except in the case that $a=b$ and the white and black vertices exactly alternate around the circle).  Some elements of this good behavior are conjectural.  For example, every cluster algebra has a collection of elements known as \emph{cluster monomials}, which are known to be linearly independent in a quite general setting that includes these cases~\cite{IKLP13}.  Fomin and Pylyavskyy conjecture that each cluster monomial corresponds to a web diagram; in fact, they have a conjectural criterion to distinguish this subset of the web diagrams: cluster monomials are conjectured to correspond to \emph{arborizable} web diagrams \cite[Conjectures 9.3 and 10.6]{fomin2016tensor}.  

Bases of cluster algebras have been of considerable interest since the genesis of cluster algebra theory going back to Fomin--Zelevinsky \cite{fomin2002cluster} and Sherman--Zelevinsky \cite{sherman2004positivity}.
Indeed, the hope that cluster algebras would give explicit combinatorial constructions of bases such as Lusztig's dual canonical basis or dual semi-canonical basis (of rings for which these are defined) was one of the initial motivations of the development of cluster algebras.  


Khovanov and Kuperberg \cite{khovanov1999web} describe a way to translate a web into a string in a three-letter alphabet.  This labelling has a natural interpretation as a particular monomial in the corresponding web invariant.  We show that there is a monomial order (in the sense of, for example, \cite{cox2007ideals}) on $\mathbb{C}[(V^*)^a \times V^b]$ such that this monomial is in fact the leading term of the corresponding web invariant.  This is of interest because it provides a very efficient algorithm for expanding an element of the ring of $\SL(V)$-invariant polynomials on $(V^*)^a\times V^b$ in terms of the web basis: in $\mathbb{C}[(V^*)^a \times V^b]$, locate the overall leading monomial and leading coefficient of $f$.  Then the web expansion of $f$ contains the corresponding web invariant with coefficient given by the leading coefficient of $f$.  Subtracting this from $f$, we obtain a new $\SL(V)$-invariant polynomial on $(V^*)^a\times V^b$ whose leading term is greater with respect to the term order.  Proceeding in this way, we obtain the web basis expansion of $f$.  We hope to apply this in future work to resolve the question of whether the web basis and the dual semicanonical basis agree, as raised, for example, in \cite{fomin2016tensor}.




\section*{Acknowledgements} The authors would like to thank Chris Fraser for helpful discussion at the start of this project. All authors received support from NSERC and the Canada Research Chairs Program. R.P. was also supported by CRM-ISM.

\section{Preliminaries}

\subsection{Tensor diagrams}
We summarize the exposition of \cite{fomin2016tensor} and refer the reader there for a more detailed treatment of this topic.

A \textit{tensor diagram} is a finite bipartite graph $D$ with a fixed proper coloring of its vertices using two colors, black and white, and with a fixed partition of its vertex set into a set bd$(D)$ of \emph{boundary vertices} and a set int$(D)$ of \emph{internal vertices} with the property that
\begin{enumerate}
\item each internal vertex is trivalent and
\item there is a fixed cyclic order on the edges incident to each internal vertex. 
\end{enumerate}

We say a tensor diagram is of \textit{type} $(a,b)$ if it has $a$ white boundary vertices and $b$ black boundary vertices. We define the \textit{multidegree} of a tensor diagram to be the sequence of degrees of the boundary vertices and note that this is defined up to rotation, see Example~\ref{ex:tensorinvariant}.

Let $V=\mathbb{C}^3$. The special linear group $\SL(V)$ acts on $V$ and $V^*$, where the latter action is defined by $(gu^*)(v)=u^*(g^{-1}(v))$ for $v\in V$, $u^*\in V^*$, and $g\in \SL(V)$. Thus $\SL(V)$ acts on the vector space
\[(V^*)^a \times V^b = V^*\times \cdots\times V^*\times V\times\cdots \times V\]
and on its coordinate ring, which is a polynomial ring in $3(a+b)$ variables. Define $R_{(a,b)}(V)$ to be the ring of $\SL(V)$-invariant polynomials on $(V^*)^a\times V^b$.
\[R_{(a,b)}(V)=\mathbb{C}[(V^*)^a\times V^b]^{\SL(V)}.\] We will sometimes refer to $R_{(a,b)}$ as the \textit{mixed invariant ring}. For $b\geq 3$, the ring $R_{(0,b)}$ is isomorphic to $\mathbb{C}[\text{Gr}_{3,b}]$---the homogeneous coordinate ring of the Grassmann manifold of $3$-dimensional subspaces in 
$\mathbb{C}^{\color{purple} b}$ with respect to its Pl\"ucker embedding---and $R_{(1,b)}$ is isomorphic to the homogeneous coordinate ring of the two-step partial flag manifold
\[\{(V_1,V_3):V_1\in \text{Gr}_{1,b},V_3\in \text{Gr}_{{\color{purple}3},b}\};\]
see \cite{fomin2016tensor} for further details.

We define the \textit{signature} of an invariant $f$ to be the word in $\{\circ,\bullet\}$ that represents the order of the $a+b$ arguments of $f$. For example, if 
\[f:V^*\times V^* \times V \times V^* \times V\times V\rightarrow \mathbb{C},\] its signature is $\sigma=[\circ,\circ,\bullet,\circ,\bullet,\bullet]$. Let $R_\sigma(V)$ denote the ring of $\SL(V)$ invariants with signature $\sigma$. If $\sigma$ contains $a$ copies of $\circ$ and $b$ copies of $\bullet$, we say $\sigma$ is of \textit{type} $(a,b)$ and note that $R_\sigma(V)\cong R_{(a,b)}(V)$.


Using coordinates, we identify $R_{(a,b)}(V)$ with the ring of $\SL_3$ invariants of collections of $a$ covectors 
\[y(v)=[y_{-1}(v) \ y_{0}(v) \ y_{1}(v)]\] and $b$ vectors
\[x(v)=\begin{bmatrix} x_{-1}(v) \\ x_0(v) \\ x_1(v)\end{bmatrix}\] labeled by a fixed collection of $a$ white and $b$ black boundary vertices on a disk.

Define a \emph{proper} edge coloring of a tensor diagram to be a labeling of all edges in $D$ by  the \emph{colors} $-1,0,1$  such that the edges incident to any internal vertex $v$ have distinct colors. A tensor diagram $D$ of type $(a,b)$ represents an $\SL(V)$ invariant in $R_{(a,b)}(V)$, denoted by $[D]$. This invariant is defined by

\[[D]=\sum_{\ell \in L} \left( \prod_{v\in\text{int}(D)} \text{sign}(\ell(v))\right) \left(\prod_{\stackrel{v\in \text{bd}(D)}{ v\text{ black}}} x(v)^{\ell(v)}\right) \left(\prod_{\stackrel{v\in \text{bd}(D)}{v\text{ white}}}y(v)^{\ell(v)}\right),\]

where
\begin{itemize}
\item $L$ is the set of all proper edge colorings of $D$,
\item sign$(\ell(v))$ denotes the sign of the permutation determined by the cyclic ordering of the labeled edges incident to $v$, where we consider the permutation $-101$ (in one-line notation) to have positive sign,
\item $x(v)^{\ell(v)}$ denotes the monomial $\prod_ex_{\ell(e)}(v)$, the product over all edges $e$ incident to $v$, and similarly for $y(v)^{\ell(v)}$. 
\end{itemize}

\begin{example}~\label{ex:tensorinvariant} 
Figure~\ref{proper-edge-colorings} shows three proper edge colorings of a tensor diagram $D$ with boundary vertices labeled. We denote the corresponding vectors and covectors by 
\[\begin{bmatrix}x_{-1,1} \\ x_{0,1} \\ x_{1,1}\end{bmatrix} \hspace{.3in}[y_{2,-1} \ y_{2,0} \ y_{2,1}] \hspace{.3in}[y_{3,-1} \ y_{3,0} \ y_{3,1}] \hspace{.3in} \begin{bmatrix}x_{-1,4} \\ x_{0,4} \\ x_{1,4}\end{bmatrix}. \]
The first proper edge coloring contributes the monomial $x_{-1,1}y_{2,1}y_{3,-1}^2x_{-1,4}x_{1,4}$, the second contributes $-x_{0,1}y_{2,0}y_{3,-1}y_{3,1}x_{-1,4}x_{1,4}$, and the third contributes
$x_{-1,1}y_{2,0}y_{3,-1}y_{3,0}x_{0,4}^2$.

\begin{center}
\begin{figure}
\includegraphics[width=6in]{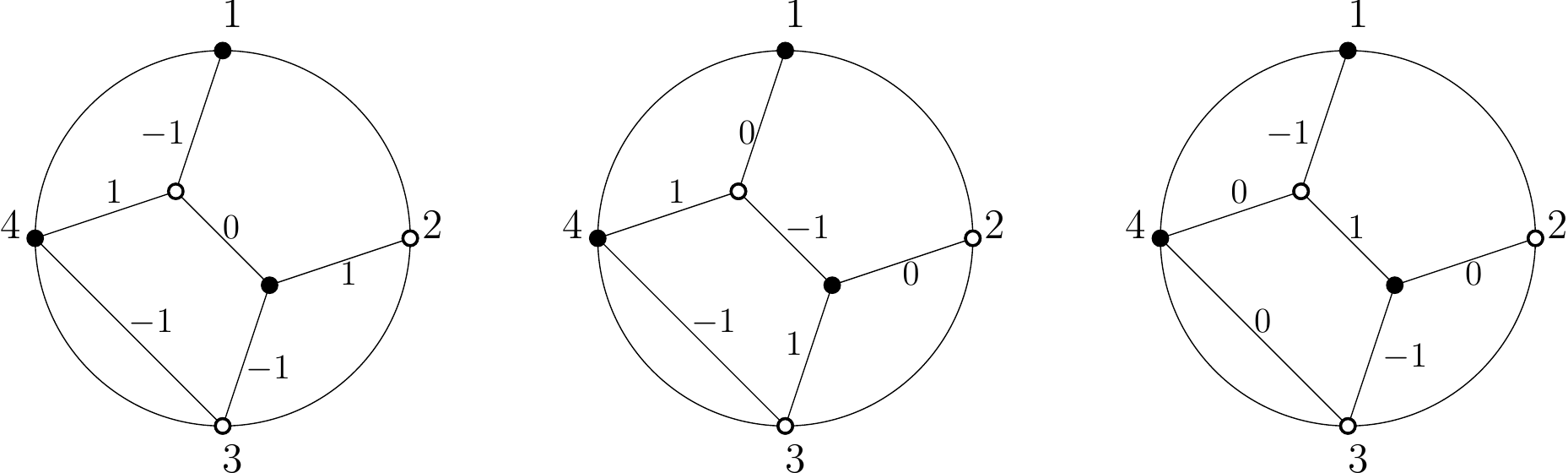}
\caption{Some proper edge colorings of the tensor diagram $D$}~\label{proper-edge-colorings}
\end{figure}
\end{center}
\end{example}


\subsection{Webs}
Webs were first defined by Kuperberg in \cite{kuperberg1996spiders}. We now discuss them using the terminology of \cite{fomin2016tensor}. 

\begin{definition}[Definition 5.1 \cite{fomin2016tensor}] A \textit{web} is a tensor diagram $D$ embedded in an oriented disk so that its edges do not cross or touch, except at endpoints. Each web is considered up to an isotopy of the disk that fixes the boundary. A web is \textit{non-elliptic} if it has no multiple edges and no 4-cycles whose vertices are all interior. The invariant $[D]$ associated with a non-elliptic web $D$ is called a \textit{web invariant}.
\end{definition}

Given a web $D$, we define its \textit{signature} to be the cyclic pattern of colors of the boundary vertices of $D$, and consider the signature up to cyclic permutation. The tensor diagram in Example~\ref{ex:tensorinvariant} is a non-elliptic web with signature $[\bullet,\circ,\circ,\bullet]$. 

\begin{theorem}[\cite{kuperberg1996spiders}] Web invariants with a fixed signature $\sigma$ of type $(a,b)$ form a linear basis in the mixed invariant ring $R_{\sigma}(V)\cong R_{(a,b)}(V)$.
\end{theorem} 

\subsection{Sign and state strings}
Let us now restrict our attention to non-elliptic webs where each boundary vertex has degree exactly 1. In \cite{khovanov1999web}, Khovanov and Kuperberg describe a bijection between such non-elliptic webs with fixed cyclic labeling of the boundary vertices and certain strings, which we now describe.

A \textit{sign string} of length $n$ is a sequence $S = (s_1, s_2, \ldots, s_n) \in \{+,-\}^n$. A \textit{state string} of length $n$ is a sequence $J = (j_1, j_2, \ldots, j_n) \in \{-1,0,1\}^n$. A \textit{sign and state string} of length $n$ is a sequence $((s_1,j_1), (s_2,j_2),\ldots,(s_n,j_n))$, where each element is a sign paired with a state. 
A sign and state string corresponds to a lattice path in the $\SL_3$ weight lattice, as follows.
 
Each vector in the sign and state string $(s_k,j_k)$ has a weight $\mu_k$. The image in Figure~\ref{fig:weightlattice} shows our convention for the weight for $(+,1)$ in red parallel to the $x$-axis, the weight for $(+,0)$ in blue 120$^\circ$ counterclockwise from the previous, and the weight for $(+,-1)$ in green. The weight for $(-,-j_k)$ is the negative of the weight for $(+,j_k)$. Given a sign and state string of length $n$, we define a path in the weight lattice of $\SL_3$ by $\pi=\pi_0,\ldots,\pi_n$, where $\pi_0$ is the origin and $\pi_k=\pi_{k-1}+\mu_k$.

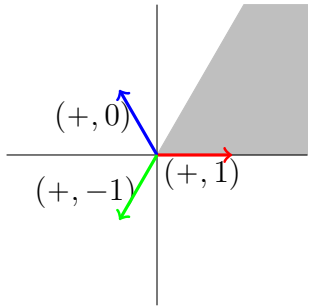
\begin{figure}[hbt!]
\begin{center}
\begin{tikzpicture}
\node at (.6,-.25) {$(+,1)$};
\node at (-.85,.5) {$(+,0)$};
\node at (-.95,-.5) {$(+,-1)$};
\draw[lightgray, top color=lightgray,bottom color=lightgray] (0,0) -- (2,0) -- (2,2) -- (1.1547,2) -- cycle;
\draw (-2,0)--(2,0);
\draw (0,2)--(0,-2);
\draw[->, very thick, red] (0,0)--(1,0);
\draw[->, very thick, blue] (0,0)--(-.5,.866);
\draw[->, very thick, green] (0,0)--(-.5,-.866);
\end{tikzpicture}
\end{center}
\caption{The weight lattice for $\SL_3$ with dominant chamber shaded gray.}
\label{fig:weightlattice}
\end{figure}

The \textit{dominant Weyl chamber} is defined as the subset of the weight lattice consisting of positive integral linear combinations of the weights for $(+,1)$ and $(-,1)$ and is shaded gray in the figure. 
We call a sign and state string \textit{dominant} if its corresponding path ends at the origin and is contained in the dominant chamber. For example, the sign and state string $(+,1),(+,1),(+,0),(-,1),(+,0),(+,-1),(-,0),(+,-1),(-,-1)$ is dominant as shown in Figure~\ref{dominant-path}. 

\begin{figure}[h!]
\includegraphics[width=2in]{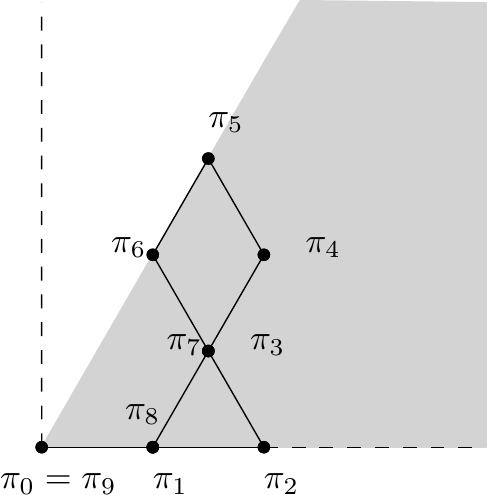}
\caption{The dominant path in the weight lattice corresponding to the sign and state string $(+,1),(+,1),(+,0),(-,1),(+,0),(+,-1),(-,0),(+,-1),(-,-1)$.}~\label{dominant-path}
\end{figure}

Given a dominant sign and state string of length $n$,  Khovanov and Kuperberg describe how to build a non-elliptic web with $n$ boundary vertices all of degree 1 by giving a series of inductive rules. To construct a web from a dominant sign and state string of length $n$, place $n$ vertices on a line segment with colors corresponding to the sign string: boundary vertex $k$ is black (resp., white) if $s_k=+$ (resp., $s_k=-$). Draw an edge stemming from each vertex, and label each edge with the corresponding state given in the sign and state string. Construct the web by successively applying an applicable growth rule from Figure~\ref{fig:growthrules}. The top row shows the rules for edges of different sign, and the bottom row shows the rules for edges of the same sign. In the first three rules in the top row, the sign of the left and right edges swap to obtain the sign of the two lower edges, while in the bottom row, the new lower edge has sign opposite of the sign of the top edges. Lastly, connect the ends of the initial line segment together so the web is contained in a disk. The first half of  Figure~\ref{proper-KK} illustrates this process. Note that the rest of the bipartite vertex coloring is determined by the colors of the boundary vertices.

\begin{figure}[h]
\raisebox{.5in}{Different:} \hspace{.3in}\begin{tikzpicture}[scale=.8] 
\draw[-] (1,-1) -- (0,0)  node [above, label=right: 1] {};
\draw[-] (1,-1) -- (0,-2) node [below, label=right:0] {};
\draw[-] (1,-1) -- (2,-1) node {};
\draw[-] (2,-1) -- (3,0) node [above, label=left:0] {};
\draw[-] (2,-1) -- (3,-2) node [below, label=left:1] {};
\end{tikzpicture}
\hspace{.3in}
\begin{tikzpicture}[scale=.8] 
\draw[-] (1,-1) -- (0,0)  node [above, label=right: 0] {};
\draw[-] (1,-1) -- (0,-2) node [below, label=right:-1] {};
\draw[-] (1,-1) -- (2,-1) node {};
\draw[-] (2,-1) -- (3,0) node [above, label=left:0] {};
\draw[-] (2,-1) -- (3,-2) node [below, label=left:1] {};
\end{tikzpicture}
\hspace{.3in}
\begin{tikzpicture}[scale=.8] 
\draw[-] (1,-1) -- (0,0)  node [above, label=right: 0] {};
\draw[-] (1,-1) -- (0,-2) node [below, label=right:-1] {};
\draw[-] (1,-1) -- (2,-1) node {};
\draw[-] (2,-1) -- (3,0) node [above, label=left:-1] {};
\draw[-] (2,-1) -- (3,-2) node [below, label=left:0] {};
\end{tikzpicture}
\hspace{.3in}
\raisebox{.5in}{\begin{tikzpicture}[scale=.8] 
\draw[-] (0,0)  to[out=-20,in=200] (2,0)  node [below, label=right:-1] {};
\draw[-] (2,0)  to[out=200,in=-20] (0,0)  node [below, label=left:1] {};
\end{tikzpicture}} \\
\vspace{.3in}
\raisebox{.5in}{Same:}\hspace{.3in} \begin{tikzpicture}[scale=1] 
\draw[-] (1,-1) -- (0,0)  node [above, label=right: 1] {};
\draw[-] (1,-1) -- (1,-2) node [below, label=right:1] {};
\draw[-] (1,-1) -- (2,0) node [above, label=left:0] {};
\end{tikzpicture}\hspace{.3in}
\begin{tikzpicture}[scale=1] 
\draw[-] (1,-1) -- (0,0)  node [above, label=right: 0] {};
\draw[-] (1,-1) -- (1,-2) node [below, label=right:-1] {};
\draw[-] (1,-1) -- (2,0) node [above, label=left:-1] {};
\end{tikzpicture}\hspace{.3in}
\begin{tikzpicture}[scale=1] 
\draw[-] (1,-1) -- (0,0)  node [above, label=right: 1] {};
\draw[-] (1,-1) -- (1,-2) node [below, label=right:0] {};
\draw[-] (1,-1) -- (2,0) node [above, label=left:-1] {};
\end{tikzpicture}
\caption{Khovanov--Kuperberg's inductive growth rules}
\label{fig:growthrules}
\end{figure}
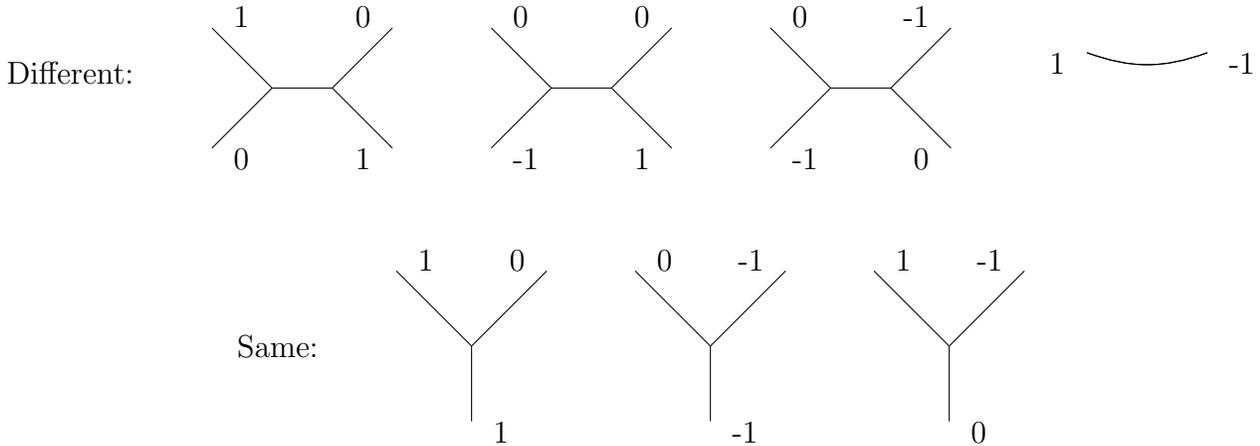

 Khovanov and Kuperberg show that if we begin with a dominant sign and state string, the algorithm successfully terminates in a web, and the resulting web does not depend on the choices made of which rules to apply.

Note that each application of the growth rules assigns a sign and state to the new half-edges that it produces at the bottom of each diagram in Figure \ref{fig:growthrules} but that in two respects, this does not amount to an assignment of a sign/state pair to each edge of the web. For one thing, the horizontal edges in the first three diagrams in the top row of Figure \ref{fig:growthrules} are not assigned a pair. Secondly, the final diagram in the top row joins two half-edges which have different states (1 and $-1$) and opposite signs.

\section{Minimal coloring}\label{min-sec}

As proven in \cite{khovanov1999web}, performing the growth algorithm starting from any dominant sign and state string will result in a non-elliptic web where each boundary vertex has degree exactly 1. Conversely, given a non-elliptic web with a fixed linear order on the boundary vertices, they also describe how to construct its corresponding sign and state string using a \textit{minimal cut path algorithm}, which is inverse to the growth rules. We do not describe this algorithm here, but use its existence to emphasize the fact that there is a one-to-one correspondence between non-elliptic webs with fixed boundary vertex order and dominant sign and state strings. Due to this correspondence, we may define the following.

\begin{definition}
Let $D$ be a non-elliptic web with sign string $S$. The \textit{KK-labeling} for $D$ is the state string  that---along with $S$---produces $D$ using the growth algorithm of Khovanov and Kuperberg described previously. In other words, the KK-labeling is a function that assigns each edge incident to a boundary vertex of $D$ to the corresponding entry in the state string for $D$.
\end{definition}

We next describe how to obtain a proper edge coloring of the non-elliptic web $D$ using colors $-1,0,1$ starting from its KK-labeling. To obtain a coloring, one need only perform the following simple steps.
\begin{enumerate}
\item Perform the Khovanov--Kuperberg growth algorithm to obtain an edge labeling consisting of a sign and state of each non-horizontal edge of the web $D$.
\item If an edge has sign $+$, its color is the same as its state.
\item If an edge has sign $-$, its color is the negative of its state.
\item If an edge is an unlabelled horizontal edge, its color is the unique color that will result in a proper edge coloring.
\end{enumerate}
Note that, when we have applied the growth rule from the rightmost diagram in the top row of Figure \ref{fig:growthrules}, as already mentioned, the resulting horizontal edge has two different sign/state labels, one at each end, but they determine the same color, so our coloring is well-defined.

The fact that the coloring obtained from the KK-labeling is a proper edge coloring using $-1,0,1$ follows easily from the inductive growth rules. This fact is also used by Petersen--Pylyavskyy--Rhoades in \cite{petersen2009promotion} and by Patrias in \cite{patrias2019promotion}. Figure~\ref{proper-KK} shows an example of the proper edge coloring starting from a KK-labeling and uses the same sign and state string as in Figure~\ref{dominant-path}.


\begin{figure}[h!]
\includegraphics[width=12cm]{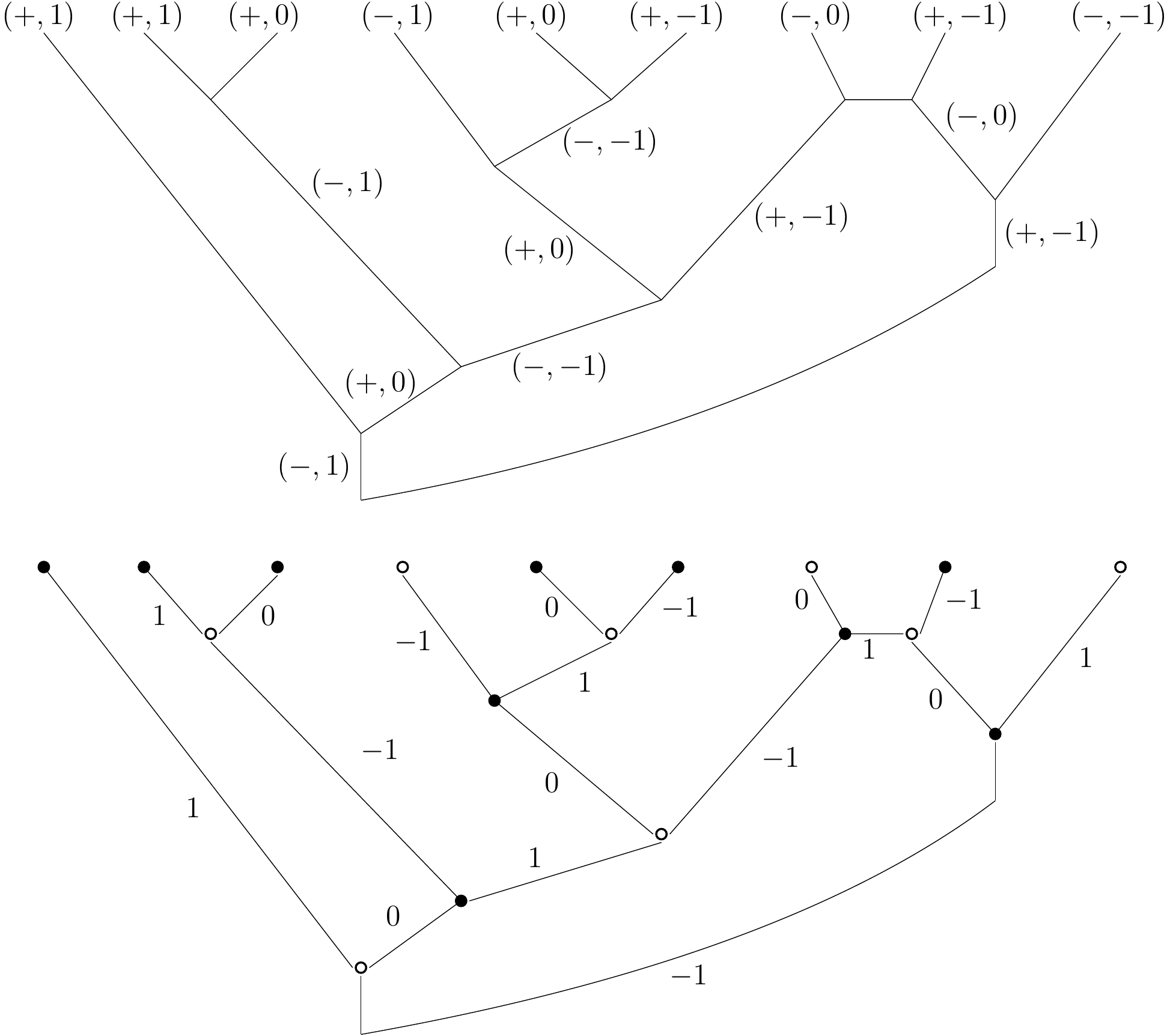}
\caption{The first image shows a sign and state string made into a web using the growth rules. The second image shows the proper edge coloring and bipartite vertex coloring obtained from the KK-labeling.}~\label{proper-KK}
\end{figure}


We next define a lexicographic order on all proper edge colorings of a non-elliptic web $D$ by linearly ordering the colors of edges incident to black boundary vertices by $1<0<-1$ and linearly ordering the colors edges incident to white boundary vertices by $-1<0<1$. Given a web $D$ with a fixed clockwise numbering of its boundary vertices $1$ to $n$, 
we may then lexicographically order the proper edge colorings of $D$ by examining the word obtained from reading the colors along the boundary vertices. For example, the coloring shown in Figure~\ref{proper-KK} corresponds to the word $(1_{\bullet},1_{\bullet},0_{\bullet}, -1_{\circ},0_{\bullet}, -1_{\bullet},0_{\circ},-1_{\bullet},1_{\circ})$, where we include the vertex color as a subscript for convenience. We see that this coloring is lexicographically smaller than a coloring with word $(1_{\bullet},1_{\bullet},-1_{\bullet},-1_{\circ},0_{\bullet},-1_{\bullet},0_{\circ},0_{\bullet},1_{\circ})$ since $0_{\bullet}<-1_{\bullet}$. Note that if two proper edge colorings have the same colors on the edges adjacent to the boundary, these colorings are equal in this lexicographic order. 

Using this lexicographic order, we may now discuss the lexicographically minimal coloring(s) for any non-elliptic web $D$. We shall see that there is a unique such coloring and it corresponds to the coloring obtained using the KK-labeling.

\begin{theorem}\label{min_thm} The proper edge coloring determined by the KK-labeling of a non-elliptic web is lexicographically minimal.
\end{theorem}

We will prove this theorem by induction on the number of vertices of the non-elliptic web. Before giving the proof, we describe the procedure we will use to reduce a web to one with strictly fewer vertices.

Let $D$ be a non-elliptic web with $n$ cyclically-labeled boundary vertices.
From the definition of a dominant path, we see that the first entry of the state string of $D$ will be 1 and the last entry will be $-1$. Given the sign and state string of $D$, $((s_1,j_1),\ldots,(s_n,j_n))$, we may thus identify some smallest $1\leq i<n$ such that $j_{i+1}\neq 1$. Since the inductive rules for the growth algorithm may be applied in any order, we know that the $(i+1)$-th boundary vertex of $D$ is connected to the $i$-th boundary vertex of $D$ in one of the ways shown in Figure~\ref{fig:reduce}.

\begin{figure}[h]
$$\begin{array}{ccccc}
\begin{tikzpicture}[scale=.8] 
\draw[-] (1,-1) -- (0,0)  node [above, label=right: $j_i$] {};
\draw[-] (1,-1) -- (0,-2) {};
\draw[-] (1,-1) -- (2,-1) node {};
\draw[-] (2,-1) -- (3,0) node [above, label=left:$j_{i+1}$] {};
\draw[-] (2,-1) -- (3,-2) node {};
\end{tikzpicture}&
& \begin{tikzpicture}[scale=1] 
\draw[-] (1,-1) -- (0,0)  node [above, label=right: $j_i$] {};
\draw[-] (1,-1) -- (1,-2) node {};
\draw[-] (1,-1) -- (2,0) node [above, label=left:$j_{i+1}$] {};
\end{tikzpicture}
& &
\raisebox{.3in}{\begin{tikzpicture}[scale=.8] 
\draw[-] (0,0)  to[out=-20,in=200] (2,0)  node [below, label=right:$j_{i+1}$] {};
\draw[-] (2,0)  to[out=200,in=-20] (0,0)  node [below, label=left:$j_i$] {};
\end{tikzpicture}}\\
(a) & & (b) & & (c)
\end{array}$$
\caption{}
\label{fig:reduce}
\end{figure}

We now define a non-elliptic web $D'$ obtained from $D$. Keeping the rest of the web the same, perform the following surgery:

\begin{itemize}
\item[(a)] If the boundary vertices $i$ and $i+1$ are connected as in situation $(a)$ of Figure~\ref{fig:reduce}, replace this section of the web as shown in Figure~\ref{Fig:trim1}, where the state and sign of the lower edges are as determined by the growth algorithm. (Delete the two boundary vertices and create two new boundary vertices.)

\begin{figure}[h]
\centering
\includegraphics[width=10cm]{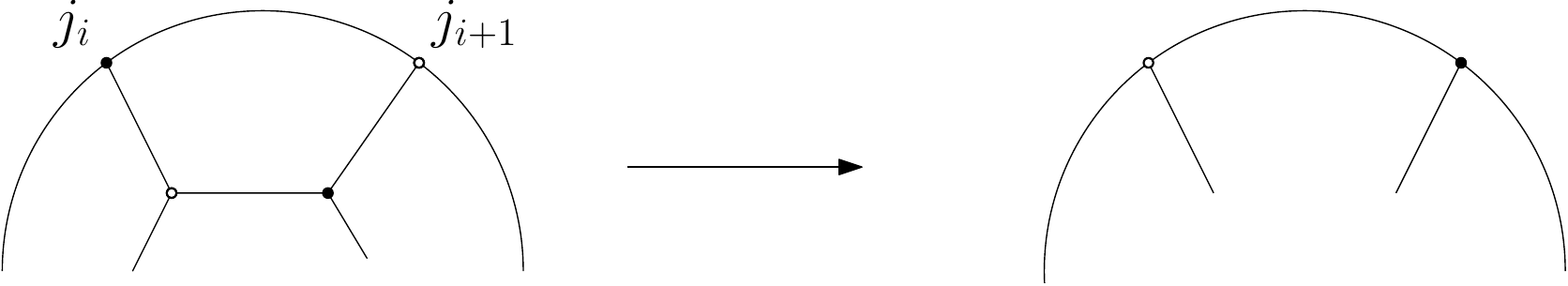}
\caption{\label{Fig:trim1}}
\end{figure}


\item[(b)] If boundary vertices $i$ and $i+1$ are connected as in situation $(b)$ of Figure~\ref{fig:reduce}, replace this section of the web as shown in Figure~\ref{Fig:trim2}. (Delete the two boundary vertices and create one new boundary vertex.)

\begin{figure}[h]
\centering
\includegraphics[width=10cm]{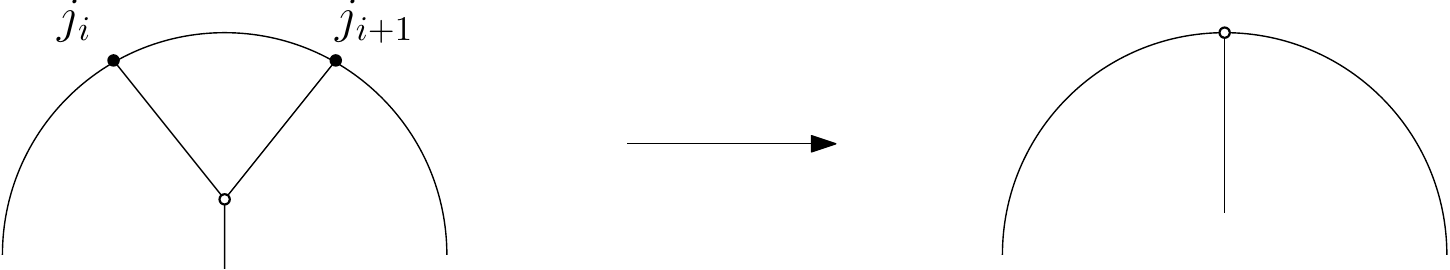}
\caption{\label{Fig:trim2}}
\end{figure}


\item[(c)] If boundary vertices $i$ and $i+1$ are connected as in situation $(c)$ of Figure~\ref{fig:reduce}, simply delete these boundary vertices and the edge between them.
\end{itemize}

We call this procedure \textit{trimming}, see Figure~\ref{Fig:trimming} for an example. We will need the following fact about this procedure. It is true in much more generality, but we will only state the version used in the proof of our main result in this section.

\begin{lemma}\label{Lemma:trimming} Let $D$ be a non-elliptic web with $n$ cyclically-labeled boundary vertices and let $D'$ be the non-elliptic web obtained from $D$ by trimming. Then the state string obtained by reading the state of each edge adjacent to the boundary of $D'$ is the KK-labeling of $D'$.
\end{lemma}
\begin{proof} This follows directly from the fact that the inductive rules for the growth algorithm may be applied in any order.
\end{proof}

\begin{figure}[h]
\centering
\includegraphics[width=13cm]{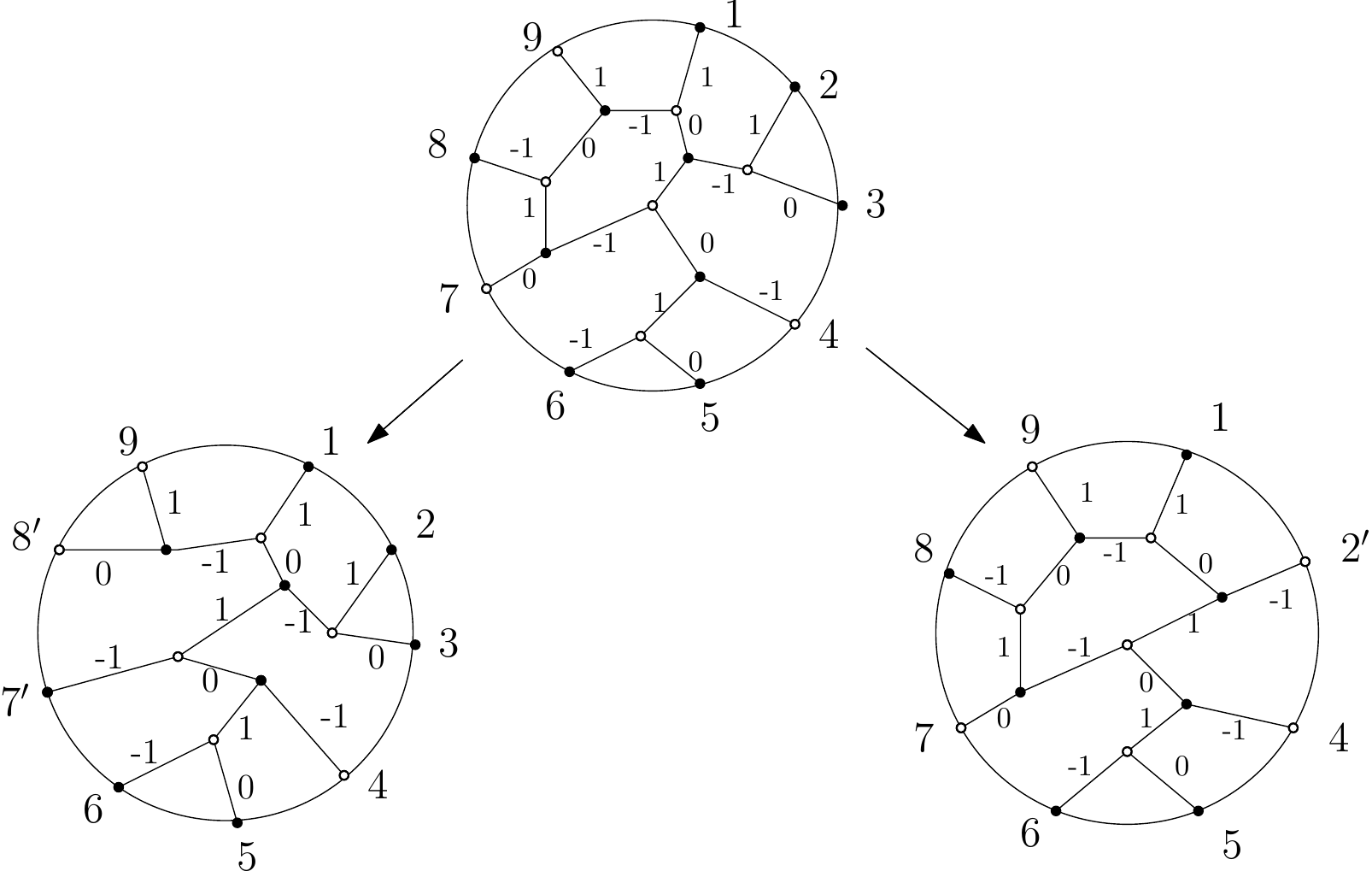}
\caption{Two different types of trimming a web. On the left side we apply type $(a)$ trimming on the vertices $7$ and $8$, and on the right side type $(b)$ trimming on the vertices $2$ and $3$. }
\label{Fig:trimming}
\end{figure}

\begin{proof}[Proof of Theorem~\ref{min_thm}]
Suppose $D$ is a non-elliptic web and that the proper edge coloring determined by the KK-labeling of any non-elliptic web with fewer vertices than  $D$ is lexicographically minimal. Let $(s_1,\dots,s_k)$ be the state string of $D$ and let 
$(j_1, \ldots, j_k)$ be its KK-labeling. We will first assume $j_{i+1} = 0$ and that $j_r =1$ for any $r \leq i$. 

Since we know that $j_{i+1} = 0$, we are in either case $(a)$ or $(b)$ in Figure~\ref{fig:reduce}. We only prove the result in the case where these edges appear in configuration $(a)$ as the proof in case $(b)$ is similar but simpler.

First, trim the web $D$ as in Figure~\ref{Fig:trim1}. By Lemma~\ref{Lemma:trimming}, the resulting web $D^\prime$ has its KK-labeling given by $((s_1,j_1),\ldots, (s_{i-1},j_{i-1}), (s_i^*,0), (s_{i+1}^*, 1), (s_{i+2},j_{i+2}), \ldots, (s_k,j_k))$ where $s_i^* \neq s_i$ and $s_{i+1}^* \neq s_{i+1}$ . By induction, we know that the KK-labeling of $D^\prime$ gives the  minimal proper edge coloring of $D^\prime$.

Now, suppose that the KK-labeling of $D$ does not give the lexicographically minimal proper edge coloring of $D$. 
It is clear that the edges of $D$ from vertices 1, 2,\ldots, $i$ are colored in a lexicographically minimal way. To see that the color of the edge from vertex $i+1$ is also lexicographically minimal, first suppose vertex $i$ is white and vertex $i+1$ is black, which means the edge incident to vertex $i$ is colored $-1$ and the edge incident to vertex $i+1$ is colored 0. 
Then if ${\ell^*}$ were another proper edge coloring of $D$ where the edge from vertex $i+1$ were colored with 1 (instead of 0) and $j_{r} = 1$ for any $r \le i$, the induced coloring on $D'$ would have the edge from vertex $i$ colored 1 and the edge from vertex $i+1$ colored $-1$. This implies that trimming $D$ with respect to the coloring $\ell^*$ would give a smaller coloring than that obtained from its KK-labeling. This contradicts Lemma~\ref{Lemma:trimming}. The argument is analogous if vertex $i$ is black and vertex $i+1$ is white. 

Since edges from vertices $1,\ldots,i+1$ are colored in a lexicographically minimal way, there must exist another strictly smaller coloring $\ell^\prime$ that differs only on edges incident to vertices $i+2$ through $n$. By trimming $D$ and keeping all remaining edge colors the same, this coloring $\ell^\prime$ restricts to a proper edge coloring of $D^\prime$ that is lexicographically smaller than the coloring from its KK-labeling. This is a contradiction.

If $j_{i+1}=-1$, the case for situation $(c)$ of Figure~\ref{fig:reduce} is trivial, and situation $(a)$ cannot occur.  Consider situation $(b)$. Suppose vertices $i$ and $i+1$ are black. If it were possible to color the edge incident to vertex $i+1$ with 0 instead of $-1$, then the color of the vertical edge below would change to $-1$. This would lead to a coloring of $D^\prime$ that is lexicographically smaller than that coming from its KK-labeling. This shows that the first $i+1$ vertices are colored in a lexicographically minimal fashion; the proof that the whole coloring is lexicographically minimal now follows the same argument as above. The case where vertices $i$ and $i+1$ are white is completely analogous.
 \end{proof}

\begin{lemma}\label{lem:onecoloring}
There is exactly one proper edge coloring of non-elliptic web $D$ whose edges incident to the boundary vertices have colors determined by the KK-labeling. 
\end{lemma}
\begin{proof}
This follows from the growth rules.
\end{proof}

\section{Monomial Orderings} 
In this section, we explain how to reinterpret Theorem \ref{min_thm} in terms of a monomial order. We begin by recalling this concept.
Our main reference for this material is~\cite{cox2007ideals}.

\begin{definition}[{\cite[Definition 2.2.1]{cox2007ideals}}] 
A \emph{monomial ordering} on a polynomial ring $\mathbb{C}[z_1,\ldots,z_n]$ is a total order $<$ on the set of monomials that satisfies the following two properties:
\begin{itemize}
    \item It is compatible with multiplication in the sense that if $z^\alpha$, $z^\beta$ and $z^\gamma$ are three monomials in  $\mathbb{C}[z_1,\ldots,z_n]$
    then $z^\alpha < z^\beta$ if and only if $z^{\alpha+\gamma} < z^{\beta+\gamma}$.
    \item It is a well-ordering, in the sense that every non-empty set of monomials has a smallest element.
    \end{itemize}
\end{definition}

One can observe that there is a monoid isomorphism between the monomials in a polynomial ring $\mathbb{C}[z_1,\ldots,z_n]$ and $\mathbb{Z}^n_{\geq 0}$, which is given by the map
\[ \mathbf{\alpha} = (\alpha_1,\cdots,\alpha_n)\mapsto z^{\mathbf{\alpha}}=z_1^{\alpha_1}\cdots z_n^{\alpha_n} \]
for any given $n$-tuple $(\alpha_1,\cdots,\alpha_n)\in \mathbb{Z}_{\geq 0}^n$.  The monoid structure on the monomials is multiplication, while on $\mathbb{Z}_{\geq 0}^n$ we use vector addition.  Hence, there is also a one-to-one correspondence between orderings on monomials compatible with multiplication and orderings on $\mathbb{Z}^n_{\geq 0}$ compatible with its own monoid structure. 


 
We now recall the definition of \emph{graded reverse lexicographic order}, or grevlex for short, which is a particular example of a monomial ordering, and which will turn out to be useful for our purposes.
  

\begin{definition}[{\cite[Definition 2.2.6]{cox2007ideals}}] 
Let $\alpha=(\alpha_1,\dots,\alpha_n)$ and $\beta=(\beta_1,\dots,\beta_n)$ be two elements in $\mathbb Z^n_{\geq 0}$. We say that $\alpha >_{\grlex} \beta$ if: \begin{itemize}
    \item $ \sum_{i=1}^n \alpha_i > \sum_{i=1}^n \beta_i$, or
    \item $\sum_{i=1}^n \alpha_i = \sum_{i=1}^n \beta_i$ and if $j$ is the smallest index such that $\alpha_j\ne \beta_j$, then $\alpha_j<\beta_j$.
\end{itemize}
For the corresponding monomials we will write $z^{\alpha}>_{\grlex} z^{\beta}$ if $\alpha >_{\grlex} \beta$.
\end{definition}


Given a polynomial ring $R=\mathbb C[z_1,\dots,z_n]$ with a monomial order, if $f$ is a non-zero element of $R$, the \emph{leading monomial} of $f$ is the monomial appearing in $f$ with non-zero coefficient which is largest with respect to the monomial order. The \emph{leading term} of $f$ is the leading monomial of $f$ multiplied by its coefficient in $f$.

Given a sign string $S\in\{+,-\}^n$, we define the ring $R_S$ to be the polynomial ring on $3n$ variables, consisting of $x_{-1,i}$, $x_{0,i}$, and $x_{1,i}$ for each $i$ with $S_i=+$ and $y_{i,-1}$, $y_{i,0}$, and $y_{i,1}$ for 
each $i$ with $S_i=-$.

Consider the graded reverse lexicographic order on monomials in $R_S$ induced by the order on the variables 
that first compares their corresponding vertex number, and subsequently, if $i$
is a black vertex, has $x_{-1,i}<x_{0,i}<x_{1,i}$, and if $i$ is a white vertex
has $y_{i,1}<y_{i,0}<y_{i,-1}$.
We may reinterpret Theorem~\ref{min_thm} as saying that if $D$ is a web with sign string $S$, such that all the boundary vertices of $D$ have degree one, then the leading term of $[D]$ under this ordering, comes from a unique proper edge coloring, which is the one determined by the KK-labeling of $D$.

This will be explored further in the next section.

\begin{example} 
Consider the web $D$ shown in Figure~\ref{fig:full_example}, and let the vectors corresponding to its vertices be \[\begin{bmatrix}x_{-1,1}\\ x_{0,1}\\x_{1,1}\end{bmatrix}\hspace{.1in}\begin{bmatrix}y_{2,-1} & y_{2,0} & y_{2,1}\end{bmatrix}\hspace{.1in}\begin{bmatrix}y_{3,-1} & y_{3,0} & y_{3,1}\end{bmatrix}\hspace{.1in}\begin{bmatrix}x_{-1,4}\\x_{0,4}\\x_{1,4}\end{bmatrix}.\] There are 12 proper edge colorings of $D$, which can be split into two groups of six as shown in the leftmost web and middle web in Figure~\ref{fig:full_example}. Terms in $[D]$ coming from the first family will have negative sign, while terms from the second family will have a positive sign. The term of $[D]$ that is in bold is the leading term in grevlex order and comes from the lexicographically minimal coloring (shown on the right in Figure~\ref{fig:full_example}) determined by the KK-labeling.
\begin{align*}[D]= \mathbf{x_{1,1}y_{2,-1}y_{3,1}x_{-1,4}}
+x_{1,1}y_{2,0}y_{3,1}x_{0,4}
-x_{1,1}y_{2,1}y_{3,-1}x_{-1,4}
-x_{1,1}y_{2,1}y_{3,0}x_{0,4}
+x_{0,1}y_{2,-1}y_{3,0}x_{-1,4}\\
-x_{0,1}y_{2,0}y_{3,-1}x_{-1,4}
-x_{0,1}y_{2,0}y_{3,1}x_{1,4}
+x_{0,1}y_{2,1}y_{3,0}x_{1,4}
-x_{-1,1}y_{2,-1}y_{3,0}x_{0,4}
-x_{-1,1}y_{2,-1}y_{3,1}x_{1,4}\\
+x_{-1,1}y_{2,0}y_{3,-1}x_{0,4}
+x_{-1,1}y_{2,1}y_{3,-1}x_{1,4}\end{align*}
\end{example}


\begin{figure}
    \centering
    \scalebox{.5}{
 \includegraphics{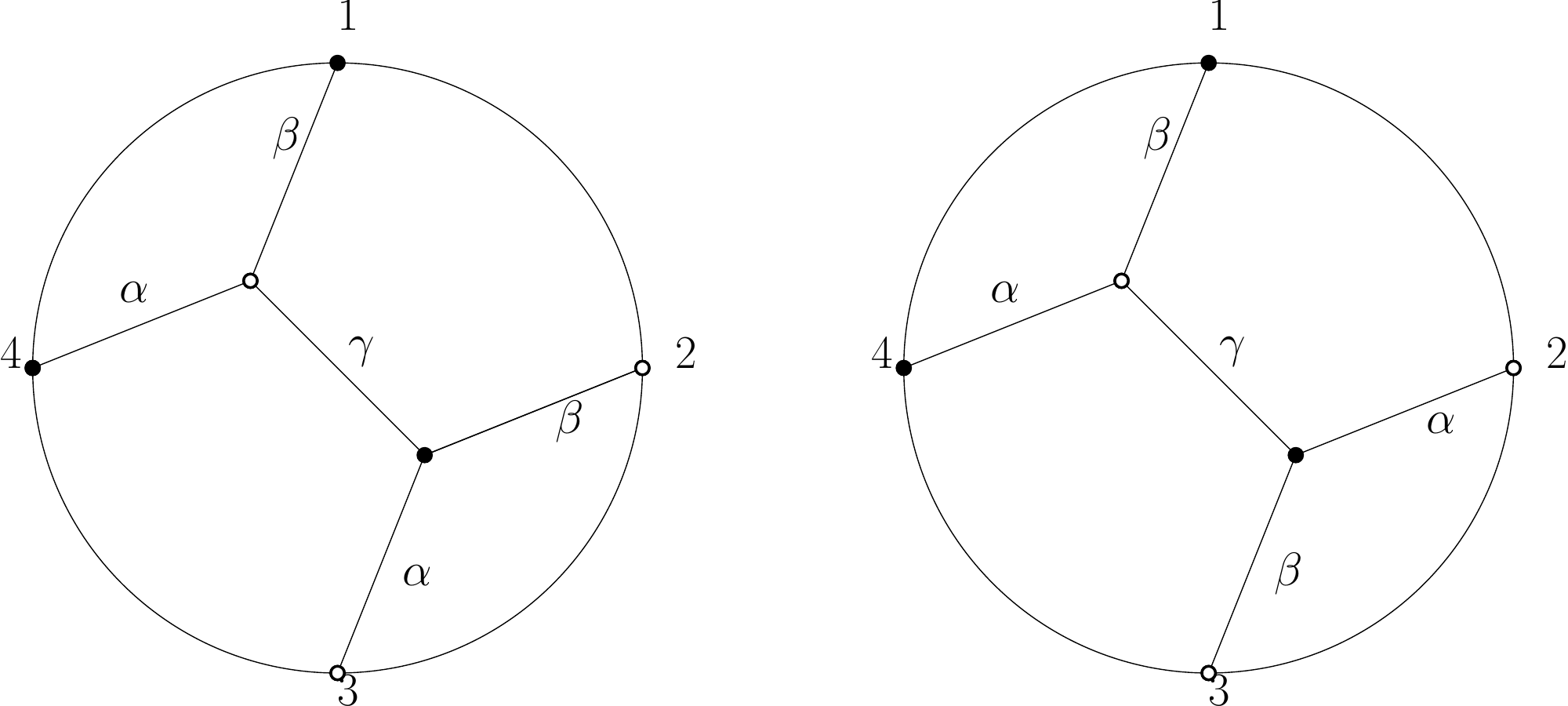}\hspace{.6in}\includegraphics{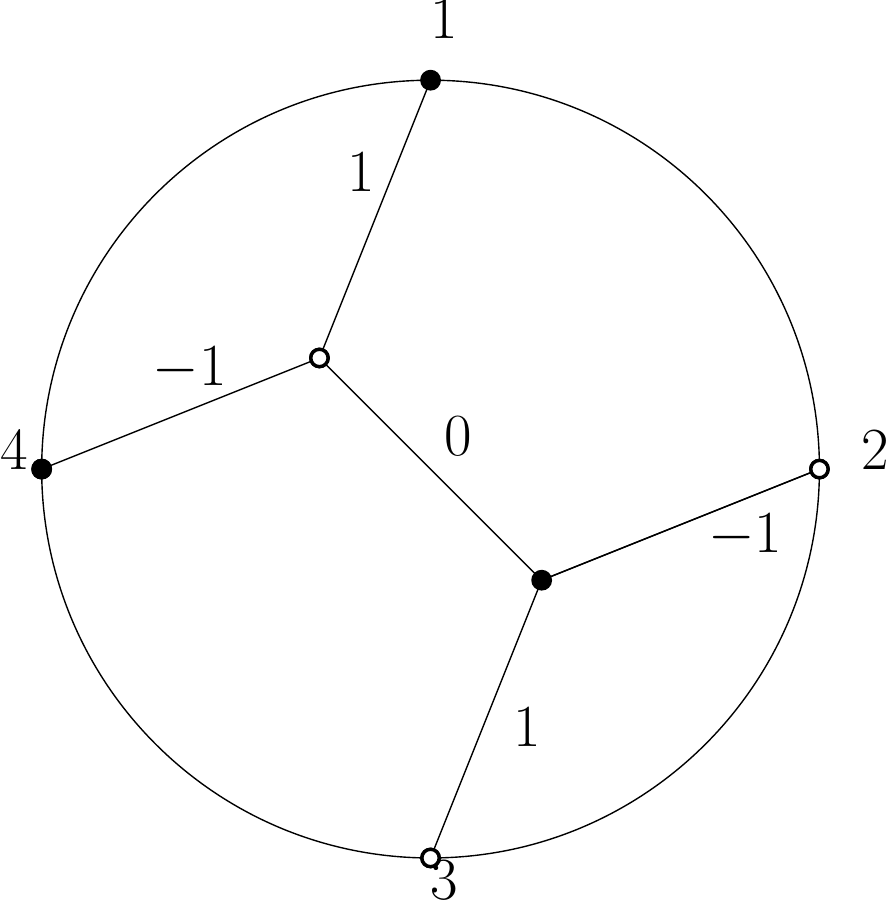} } 
    \caption{The leftmost and middle images show the two families of proper edge coloring of the web. The image on the right shows its minimal coloring.}
    \label{fig:full_example}
\end{figure}

\section{Extension to general webs}\label{sec:generalwebinvariants}

We now show how to extend our results from the previous section to webs whose boundary vertices may have degree greater than 1. Again, let $S\in\{+,-\}^n$ be a sign string, and let $R_S$ be the corresponding polynomial ring in $3n$ variables. We again take graded reverse lexicographic order on $R_S$.

Let $D$ be any non-elliptic web with with $n$ cyclically-labeled boundary vertices and multidegree $(d_1,\ldots,d_n)$, where $d_i$ is the degree of boundary vertex $i$. Define the \textit{unclasping} of $D$, denoted $\tilde{D}$, to be the non-elliptic web obtained from $D$ by replacing each boundary vertex $i$ with $d_i$ degree 1 vertices of the same color as vertex $i$, each serving as an endpoint of an edge formerly adjacent to vertex $i$, see Figure~\ref{fig:unclasping}. Then $\tilde{D}$ is a web where each boundary vertex has degree 1, and we can find its KK-labeling as in Section \ref{min-sec}. This KK-labeling corresponds to a proper edge coloring of $\tilde{D}$ and thus a proper edge coloring of $D$. There is therefore a term in $[D]$ coming from this proper edge coloring. 


\begin{figure}[htbp!]
\includegraphics[width=11cm]{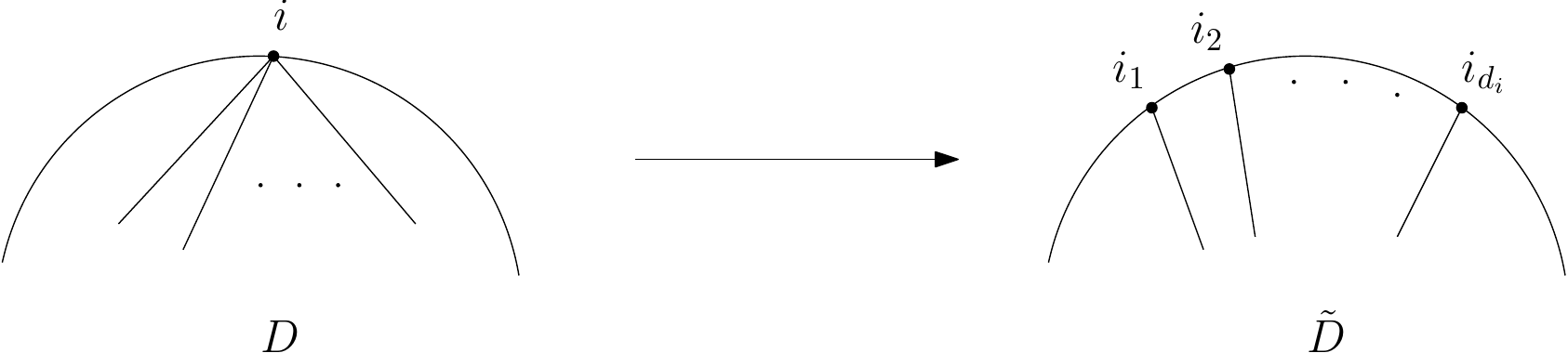}
\caption{The unclasping of a web.}\label{fig:unclasping}
\end{figure}

\begin{theorem}
Let $D$ be a nonzero non-elliptic web. The monomial coming from the KK-labeling of the unclasping $\tilde{D}$ is the leading term in the web invariant $[D]$ with respect to the graded reverse lexicographic order, and it appears with coefficient $\pm1$. 
\end{theorem}

\begin{proof}

Let $\tilde \ell$ be the proper edge coloring of $\tilde{D}$ induced by its KK-labeling. This proper edge coloring gives rise to a proper edge coloring of $D$, which we call $\ell$. Let $p$ denote the term of $[D]$ coming from $\ell$. We wish to show that all other proper colorings of $D$ give rise to monomials that are strictly less than $p$ in the graded reverse lexicographic order. 

Let $q$ be the term of $[D]$ corresponding to a different coloring $\ell'$, and let $\tilde{\ell}'$ be the corresponding coloring of $\tilde{D}$. Let $i$ be the first vertex of $D$ where $\ell$ and $\ell'$ differ. Suppose to begin with that $i$ is black.  Let ${i_1},\ldots,{i_k}$ denote the vertices and $e_{i_1}, \ldots, e_{i_k}$ the edges in $\tilde{D}$ coming from $i$.
We first show that the sequence of edge labels $(\tilde{\ell}(e_{i_1}), \ldots, \tilde{\ell}(e_{i_k}))$ is weakly increasing. Suppose to the contrary that $\tilde{\ell}(e_{i_j}) > \tilde{\ell}(e_{i_{j+1}})$ for some $j \in \{1, \ldots, k-1\}.$ We can reconstruct the web using the web growth rules (see Figure~\ref{fig:growthrules}). Starting by applying the rule to $i_j$ and $i_{j+1}$, we see that we would obtain a web such that, when it was clasped again, it would have a double edge, violating the assumption that $D$ is a non-elliptic web.

Consider the first edge of $\tilde{D}$ in clockwise order that is differently colored in $\tilde{\ell}$ and $\tilde{\ell}'$, and find the corresponding edge in $D$. If its color in $\ell'$ were greater than in $\ell$, the coloring $\tilde \ell'$ of $\tilde{D}$ would give rise to a term in $[\tilde{D}]$ greater than that coming from $\tilde{\ell}$, which is impossible. Thus, the first change converts a $0$ into a $-1$, or else a $1$ into a $0$ or a $-1$. In either case, even though we do not know how the subsequent edges connected to vertex $i$ (or later vertices) are changed, it follows that $p>q$ in graded reverse lexicographic order, as desired.

A similar argument applies if $i$ is white, where then $(\tilde \ell(e_{i_1}), \ldots, \tilde \ell(e_{i_k}))$ is weakly decreasing.
\end{proof}

\begin{example} 
The grevlex leading term of the web on the left in Figure \ref{Fig:UnclaspWeb} is $x_{1,1}y_{2,-1}x_{0,3}x^2_{-1,4}$. The grevlex leading leading term of the web on the right is $x_{1,1}y_{2,-1}x_{0,3}x_{-1,4_1}x_{-1,4_2}$.
\begin{figure}
    \centering
    \scalebox{.5}{
 \includegraphics{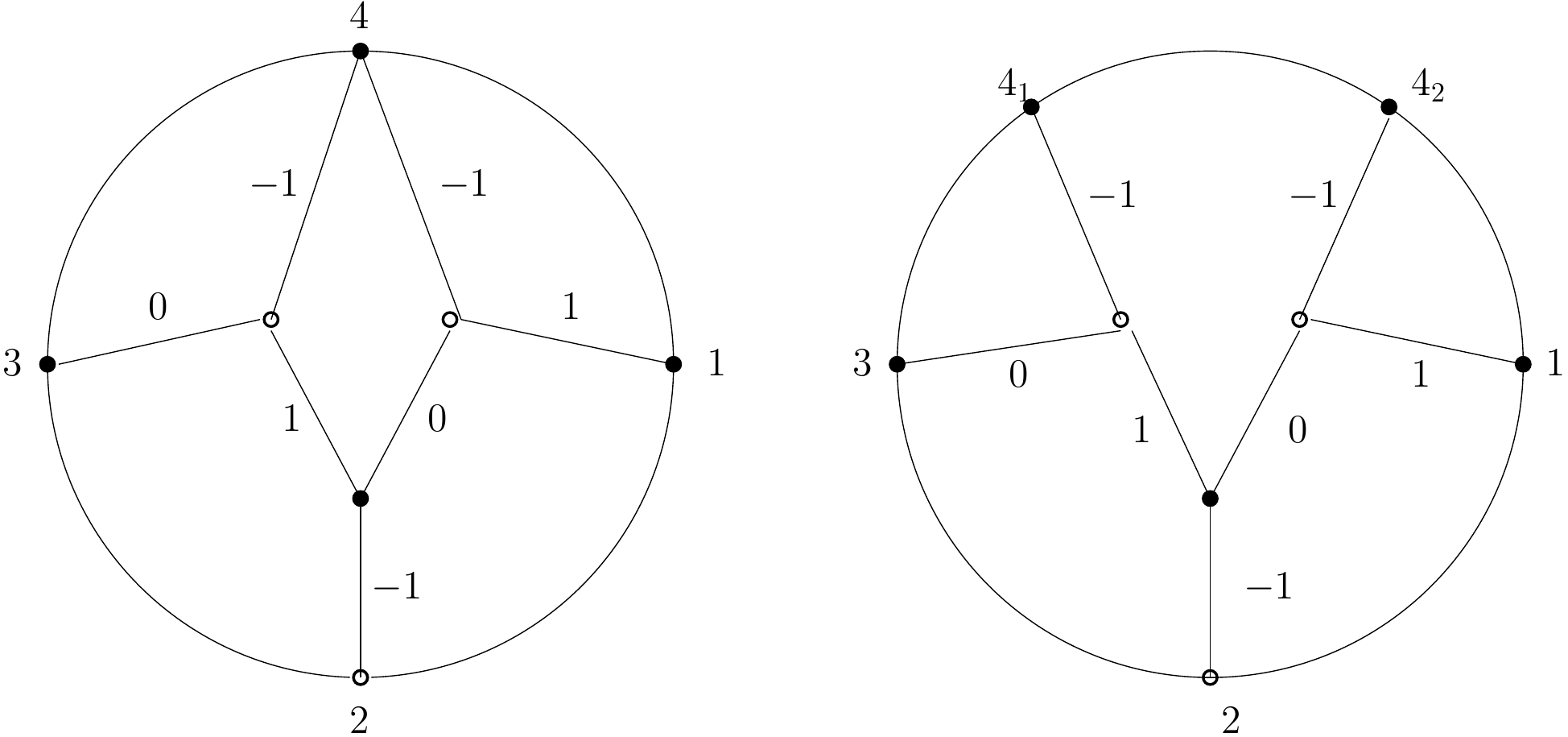}} 
    \caption{A web and its  unclasping.}
    \label{Fig:UnclaspWeb}
\end{figure}
\end{example}

\bibliographystyle{plain}
\bibliography{main}

\end{document}